\documentclass[12pt,reqno,twoside]{amsart}

\usepackage[english]{babel}
\usepackage{stmaryrd}
\usepackage{dsfont}
\usepackage[symbol*,ragged]{footmisc}
\usepackage[colorlinks,linkcolor=blue,anchorcolor=blue,citecolor=blue,urlcolor=blue]{hyperref}
\usepackage{color,xcolor}

\usepackage{geometry}
\usepackage{fancyhdr}
\usepackage{amssymb}
\usepackage{amsmath}
\usepackage{mathrsfs}
\usepackage{amsfonts}
\usepackage{epsfig}

\usepackage{amsthm}
\usepackage{amsxtra}
\usepackage{bbding}
\usepackage{epsfig}
\usepackage{graphicx}
\usepackage{latexsym}
\usepackage{mathbbol}
\usepackage{bbold}

\usepackage{pifont}
\usepackage{wasysym}
\usepackage{skull}
\usepackage{float}

\DeclareSymbolFontAlphabet{\mathbb}{AMSb}
\DeclareSymbolFontAlphabet{\mathbbol}{bbold}

\usepackage{amscd}
\usepackage[all]{xy}
\allowdisplaybreaks[4]
\usepackage{setspace}

\theoremstyle{plain}
\newtheorem{theorem}{\normalfont\scshape Theorem}[section]
\newtheorem{corollary}{\normalfont\scshape Corollary}[section]
\newtheorem{proposition}{\normalfont\scshape Proposition}[section]
\newtheorem{lemma}[proposition]{\normalfont\scshape Lemma}
\newtheorem*{corollary*}{\normalfont\scshape Corollary}
\newtheorem{remark}{\normalfont\scshape Remark}

\theoremstyle{remark}
\newtheorem*{remark*}{\normalfont\scshape Remark}

\newcommand{\Kl}{\mathrm{Kl}}
\let\ve=\varepsilon

\newcommand{\mods}[1]{\,(\mathrm{mod}\;{#1})}
\newcommand{\sumstar}{\sideset{}{^*}\sum}

\numberwithin{equation}{section}
\addtocounter{footnote}{1}

\renewcommand{\footnoterule}{
	\kern -3pt
	\hrule width 2.5in height 0.4pt
	\kern 3pt}

\makeatletter
\@ifundefined{MakeUppercase}{}{}
\makeatother

\begin{document}
\title[On the exponent of distribution for convolutions of $\operatorname{GL}(2)$ coefficients to smooth moduli]
{On the exponent of distribution for convolutions of $\operatorname{GL}(2)$ coefficients to smooth moduli}
	
\author[Rongjie Yin]{Rongjie Yin}
\address{Data science institute, Shandong University, Jinan 250100, People's Republic of China}
\email{rongjie.yin@mail.sdu.edu.cn}

\author[Tengyou Zhu]{Tengyou Zhu}
\address{School of Mathematics, Shandong University, Jinan 250100, People's Republic of China}
\email{zhuty@mail.sdu.edu.cn}


\begin{abstract}
Let $(\lambda_f(n))_{n\geqslant1}$ be the Hecke eigenvalues of a holomorphic cusp form $f$. We prove that the exponent of distribution of $\lambda_f*1$ in arithmetic progressions is as large as $\frac{1}{2}+\frac{1}{70}$ when the modulus $q$ is square-free and has only sufficiently small prime factors.
\end{abstract}

\maketitle

\section{Introduction and the main result}
Given an arithmetical function $f:\mathbb{N}\to \mathbb{C}$, we expect that
\begin{equation}\label{first}
\sum_{\substack{n\leqslant X\\n\equiv a\mods{q}}}f(n)\sim \frac{1}{\varphi(q)}\sum_{\substack{n\leqslant X\\(n,q)=1}}f(n),
\end{equation}
for each $(a,q)=1$. It is an interesting problem in number theory to show that the above asymptotic holds for $q$ as large as possible. To this end, we call a positive number $\theta$ a distribution exponent for $f$ restricted to a set $\mathcal{Q}$ of moduli, if for any $q\in\mathcal{Q}$ with $q\leqslant X^{\theta-\varepsilon}$ and for each residue class 
$a$ modulo 
$q$ with $(a,q)=1$, the asymptotic formula
\begin{equation}\notag
\sum_{\substack{n\leqslant X\\n\equiv a\mods{q}}}f(n)-\frac{1}{\varphi(q)}\sum_{\substack{n\leqslant X\\ (n,q)=1}}f(n)\ll\frac{X}{q}(\log X)^{-A}
\end{equation} 
holds for any $A>0$ and $X\geqslant 2$.

For the very important von Mangoldt function $\Lambda(n)$, the classical Siegel-Walfisz theorem implies that the above asymptotic hold for $q\leqslant (\log X)^{B}$, where $B=B(A)>0$ depends on $A$, whereas the Generalized Riemann Hypothesis (GRH) predicts the formula holds for $q\leqslant X^{\frac{1}{2}-\varepsilon}$. The celebrated Bombieri-Vinogradov theorem confirms this prediction on average over the moduli.

Another important class of examples comes from the $k$-th divisor function
\begin{equation*}
d_k(n):=\sum_{\substack{n_1\cdots n_k=n\\n_1,\dots,n_k\in \mathbb{Z^{+}}}}1.
\end{equation*}
For each fixed $k\geqslant1$, we are interested in the admissible value of $\theta_k$ being as large as possible so that for any $\varepsilon>0$, the bound
\begin{equation*}
\sum_{\substack{n\leqslant X\\n\equiv a\mods{q}}}d_k(n)-\frac{1}{\varphi(q)}\sum_{\substack{n\leqslant X\\ (n,q)=1}}d_k(n)\ll\frac{X}{q}(\log X)^{-A}
\end{equation*}
holds for all $q\leqslant X^{\theta_k-\varepsilon}$ and $A>0$. It is widely believed that one can take $\theta_k=1$ for each $k\geqslant1$. Relying on Weil's bound for Kloosterman sums, Selberg \cite{Selberg} and Hooley \cite{Hooley} independently obtained that $\theta_2=\frac{2}{3}$ is admissible. For $k=3$, this problem attracted the attention of many mathematicians. At first, Friedlander and Iwaniec \cite{FI1985} proved that one can take $\theta_3=\frac{1}{2}+\frac{1}{230}$, and this was later improved to $\frac{1}{2}+\frac{1}{82}$ by Heath-Brown \cite{Heath-Brown}. In the case of prime moduli, Fouvry, Kowalski and Michel \cite{FKM2015} proved that
$\theta_3=\frac{1}{2}+\frac{1}{46}$ is admissible by appealing to their bilinear forms of Kloosterman sums with smooth coefficients (see \cite[Theorem 1.16]{FKM2014}). More tools are available to address the problem if conditions are imposed on the prime factorization of $q$. For instance, Irving \cite{Irving}, Xi \cite{Xi}, and Sharma \cite{Sharma} improved the results. So far, the best result is due to Sharma \cite{Sharma}, who obtained $\theta_3=\frac{1}{2}+\frac{1}{30}$.

Let $f$ be a holomorphic cusp form of level $1$ with Hecke eigenvalues $\lambda_f(n)$, normalized so that $|\lambda_f(n)|\leqslant d(n)$, where $d(n)$ is the classical divisor function. For $n\geqslant1$, we define
\begin{equation}\label{定义}
    (\lambda_f*1)(n) :=\mathop{\sum\sum}_{ml=n}\lambda_f(m).
\end{equation}
The exponent of the distribution of $\lambda_f*1$ is also an interesting problem. Kowalski, Michel and Sawin \cite{KMS2017} proved that the exponent is $\frac{1}{2}+\frac{1}{102}$ when $q$ ranges over prime moduli.

In this paper, we focus on estimating the exponent for $\lambda_f*1$ when $q$ ranges over smooth moduli. Our main result is the following Theorem \ref{main theorem}, which achieves a larger exponent than in the case of the prime moduli treated in \cite{KMS2017}.

\begin{theorem}\label{main theorem}
Let $q$ be a square-free integer and have only sufficiently small prime factors, and let $a\geqslant1$ be an integer such that $(a,q)=1$. Let $X\geqslant1$ satisfy
\begin{equation*}\label{q=X}
 q\leqslant X^{\frac{18}{35}-\varepsilon}=X^{\frac{1}{2}+\frac{1}{70}-\varepsilon}.
\end{equation*}
 Then we have
\begin{equation*}\label{主估计}
\sum_{\substack{n\leqslant X\\n\equiv a\mods{q}}}(\lambda_f*1)(n)-\frac{1}{\varphi(q)}\sum_{\substack{n\leqslant X\\(n,q)=1}} (\lambda_f*1)(n)\ll_{\varepsilon,A}\frac{X}{q}(\log X)^{-A},
\end{equation*}
 where $\varepsilon>0$ is arbitrarily small and $A>0$ is arbitrarily large.
\end{theorem}

The main novelty of Theorem \ref{main theorem} lies in establishing the exponent of distribution $\frac{1}{2}+\frac{1}{70}-\varepsilon$ for the convolution $\lambda_f*1$ over smooth moduli $q$, a result that significantly improves upon the exponent $\frac{1}{2}+\frac{1}{102}$ obtained by Kowalski, Michel and Sawin \cite{KMS2017} for prime moduli. By exploiting the favorable factorization structure of the modulus and employing Weyl differencing, we transform the original problem into estimating Kloosterman sums. 

Furthermore, through a delicate analysis of exponential sums over prime components (Lemma \ref{lem:S-bounds}), combined with the Poisson summation formula and the Voronoi summation formula, we obtain a nontrivial bound for the bilinear forms in Theorem \ref{thmBKs}. This approach not only circumvents the reliance on algebraic geometry methods but also provides a new technical pathway for handling a wider class of moduli and automorphic form coefficients.
\begin{remark}
There is an analytic analog to this question. Huang, Lin and Wang \cite{HLW} proved the asymptotic formula
\begin{equation*}
    \sum_{n\leqslant X}(\lambda_f*1)(n)=L(1,f)X+O(X^{\frac{1}{2}-\frac{4}{739}+\varepsilon}).
\end{equation*}
\end{remark}
$\textbf{Notation.}$ 
In this
paper, we denote $e(z) = e^{2i\pi z}$ for $z\in \mathbb{C}$. For $n \geqslant 1$ and for an integrable function $f : \mathbb{R}^n \to \mathbb{C}$, we denote by
\begin{equation*}
    \widehat{f}(\xi) = \int_{\mathbb{R}^n} f(t) e(-\langle t, \xi \rangle) \, \mathrm{d}t
\end{equation*}
its Fourier transform, where $\langle \cdot, \cdot \rangle$ is the standard inner product on $\mathbb{R}^n$.

If $q \geqslant 1$ is a positive integer and $K : \mathbb{Z} \longrightarrow \mathbb{C}$ is a periodic function with period $q$, its Fourier transform is the periodic function $\widehat{K}$ with period $q$ defined on $\mathbb{Z}$ by
\begin{equation*}
    \widehat{K}(n) = \frac{1}{\sqrt{q}} \sum_{h \mods q} K(h) e\left( \frac{hn}{q} \right).
\end{equation*}

Given a prime $p$ and a residue class $a$ modulo $p$ with $(a,p)=1$, we denote by $\bar{a}$ the inverse of $a$ modulo $p$.

\section{Bilinear forms with Kloosterman sums}
Let $q\geqslant1$ be an integer and let $K(n)$ be a complex valued arithmetic function that is $q$-periodic. A recurring problem in analytic number theory is to understand how such functions correlate with other natural arithmetic functions $f(n)$, where $f$ could
be the characteristic function of an interval, or that of the primes, or the Fourier coefficients of some automorphic forms. When facing such problems, one is often led to the problem of bounding non-trivially some bilinear forms
\begin{equation}\label{B()}
B(\alpha,\beta;K)=\sum_{m}\sum_{l}\alpha_m\beta_lK(ml),
\end{equation}
where the coefficients $\alpha=(\alpha_m)_{m\leqslant M}$ and $\beta=(\beta_l)_{l\in \mathcal{L}}$ are supported in intervals of length $M$ and $L$, respectively. We denote
\begin{equation*}
\|\alpha\|_2 := \Bigl{(}\sum_m |\alpha_m|^2\Bigl{)}^{1/2}
\end{equation*}
the $\ell^{2}$ norm.
One of the main objectives is to improve upon the trivial bound
\begin{equation*}
\|K\|_{\infty}\|\alpha\|_{2}\|\beta\|_{2}(ML)^{\frac{1}{2}}
\end{equation*}
for ranges of $M$ and $L$ that are as small compared to $q$ as possible. Indeed, this uniformity is often more important than the size of the saving over the trivial bound.

In particular, when dealing with problems related to the analytic theory
of automorphic forms, one is often faced with the case where $K(ml)$ is a hyper-Kloosterman sum $\mathrm{Kl}_k(n;q)$, which is defined, for $k\geqslant2$, by
\begin{equation*}
\mathrm{Kl}_k(n;q) :=\frac{1}{q^{\frac{k-1}{2}}}\mathop{\sum \cdots \sum}_{\substack{x_1, \dots, x_k \mods{q} \\ x_1 \dots x_k \equiv n \mods{q}}} e\Bigl{(}\frac{x_1 + \dots + x_k}{q}\Bigl{)}.
\end{equation*}
It is the classical Kloosterman sum when $k=2$,
and the celebrated Weil's bound gives
\begin{equation*}
    \sum_{x\in (\mathbb{Z}/q\mathbb{Z})^\times}e\Bigl{(}\frac{ax+bx^{-1}}{q}\Bigl{)}\ll \gcd(a,b,q)^{\frac{1}{2}}q^{\frac{1}{2}+o(1)}.
\end{equation*}

In our paper, for $K=\mathrm{Kl}_3(aml; q)$, when $q$ admits a favorable factorization, we establish the estimate for the bilinear form via the following theorem.

\begin{theorem}\label{thmBKs}
Let $M,N\geqslant1$, and let $s,q$ be squarefree numbers with $s\mid q$, let $a\in \mathbb{Z}$ be coprime to $q$, and let $\alpha_{m}$ and $\beta_l$ be defined as in \eqref{B()}, and they are complex numbers supported on $[M,2M]$ and $[L,2L]$ respectively, and $|\alpha_m|\leqslant1$, $|\beta_l|\leqslant1$.
Then we have
\begin{equation}\label{eqBKs}
\begin{aligned}
&\sum_{M \leqslant m \leqslant 2M} \sum_{L \leqslant l \leqslant 2L} \alpha_m \beta_l\mathrm{Kl}_3(a m l; q) \\
&\qquad \ll q^{\varepsilon} ML\bigl(M^{-\frac{1}{2}}s^{\frac{1}{2}} + q^{-\frac{1}{4}}s^{\frac{1}{4}} + L^{-\frac{1}{2}}q^{\frac{1}{4}}s^{-\frac{1}{4}}\bigr).
\end{aligned}
\end{equation}
\end{theorem}

An application of the Cauchy--Schwarz inequality gives
\begin{align}\label{eqthm5BM}
\bigg|\sum_{M\leqslant m\leqslant 2M}\sum_{L\leqslant l \leqslant 2L} \alpha_m \beta_l \mathrm{Kl}_3(aml;q)\bigg|^2\ll\|\beta_l\|^{2}\sum_{L\leqslant l \leqslant 2L}\Bigl|\sum_{M\leqslant m\leqslant 2M}\alpha_m\mathrm{Kl}_3(aml;q)\Bigr|^2,
\end{align}
which reduces the problem to estimating the sum
\begin{align}\label{eqSumKlo}
\sum_{L\leqslant l \leqslant 2L}\Bigl|\sum_{M\leqslant m\leqslant 2M}\alpha_m\mathrm{Kl}_3(aml;q)\Bigr|^2.
\end{align}
The analysis of \eqref{eqSumKlo} bifurcates according to the factorization structure of $q$.

\begin{proof}[Proof of Theorem \ref{thmBKs}]
We begin the proof of Theorem \ref{thmBKs} by applying the following general differencing lemma (see \cite[Proposition 21]{BM2015}).  
\begin{lemma}\label{lem:Weyl-differencing}
Let $b, b_{1i}, b_{2i} : \mathbb{Z} \to \mathbb{C}$ (with $1 \leqslant i \leqslant I$), $r_2 \in \mathbb{N}$, and $R_2 \in \mathbb{R}$ satisfy
\[
b(k) = \sum_{i=1}^{I} b_{1i}(k) b_{2i}(k) \qquad (k \in \mathbb{Z}),
\]
and
\[
b_{2i}(k + r_2) = b_{2i}(k), \qquad |b_{2i}(k)| \leqslant R_2 \qquad (k \in \mathbb{Z}, \; 1 \leqslant i \leqslant I).
\]
Assume further that each $b_{1i}$ is supported on a finite interval, and let $H \in \mathbb{N}$. Then
\[
\Bigl| \sum_{k} b(k) \Bigr|^{2} \leqslant H r_2 R_2^{2} I \sum_{i=1}^{I} \sum_{\substack{k_1, k_2 
\\ k_1 \equiv k_2 \mods{H r_2}}} b_{1i}(k_1) \overline{b_{1i}(k_2)}.
\]
\end{lemma}

Suppose that $q = r_1 r_2$ with
\begin{equation}\label{q-factorization}
 (r_1, r_2) = 1.
\end{equation}
Let
\begin{equation}\label{r1-factorization}
r_1 = \prod_{j=1}^{J} p_j
\end{equation}
be the factorization of $r_1$ into primes. Define
\[
Q_j = \frac{r_1}{p_j}, \qquad Q_j \bar{Q}_j \equiv 1 \pmod{p_j}.
\]
Let $a$ be an integer with $(a, r) = 1$ as in Theorem~\ref{thmBKs}. Define
\begin{equation}\label{b1b2}
b_2(m) = \mathrm{Kl}_3(a m l \bar{r}_1^3 ; r_2),\qquad
b_1(m) = \alpha_m \mathrm{Kl}_3(a m l \bar{r}_2^3 ; r_1).
\end{equation}
By the twisted multiplicativity of Kloosterman sums, we have
\[
b_1(m) b_2(m) = \alpha_m \Kl_3(a m l; q).
\]
Since $(a, r) = 1$, we have according to Weil's bound
$$
b_2(m)\leqslant R_2 :=\tau(r_2).
$$
Let $V$ be a smooth, non‑negative function supported in $[1/2,3]$ that is identically $1$ on $[1,2]$. Then it suffices to estimate
\begin{equation}\label{eq:Sigma}
\Sigma := \sum_{l \geqslant 1 }
V\Bigl( \frac{l}{L} \Bigr)
\biggl| \sum_{M \leqslant m \leqslant 2M} \alpha_m \Kl_3(a m l; q) \biggr|^{2}.
\end{equation} 
We factor $q = r_1 r_2$ as in~\eqref{q-factorization} and~\eqref{r1-factorization}, 
and then choose a natural number $H$ that satisfies $H \mid r_1$.
Applying Lemma~\ref{lem:Weyl-differencing} with $b_1, b_2$ 
as in~\eqref{b1b2},  we obtain
\[
\Sigma  \ll r^\varepsilon H r_2^2  \sum_{l \geqslant 1} V\left(\frac{l}{L} \right) 
\sum_{\substack{M \leqslant m_1, m_2 \leqslant 2M \\ m_1 \equiv m_2 \mods {H r_2} 
}} b_1 (m_1) \overline{b_1 (m_2)}.
\]
Diagonal terms contribute
\[
\ll (qM)^\ve H r_2^2 L r_1 \sum_{M\leqslant m\leqslant 2M} |\alpha_m|^2 \ll (qLM)^\ve H r_2^2 L M r_1.
\]
We find that
\begin{equation}\label{Sigma-ll}
\Sigma \ll (qLM)^\ve H r_2^2 L M r_1 + q^\varepsilon H r_2^2 
\sum_{\substack{M \leqslant m_1 \neq m_2 \leqslant 2M \\ m_1 \equiv m_2 \mods {H r_2} 
}} \left| \Sigma (m_1, m_2, r_1) \right|,
\end{equation}
where
\[
\Sigma(m_1, m_2, r_1) := \sum_{l} V \left( \frac{l}{L} \right) 
\mathrm{Kl}_3(a m_1 l \bar{r}_2^3; r_1) \overline{\mathrm{Kl}_3(am_2 l \bar{r}_2^3; r_1)}.
\]

By Poisson summation we have
\begin{equation}\label{after-Poisson}
\Sigma(m_1, m_2, r_1) = \frac{L}{r_1} \sum_{h \in \mathbb{Z}}
\widehat{V} \left( \frac{hL}{r_1} \right) \sum_{l \mods{r_1}} 
\Kl_3(a m_1 l \bar{r}_2^3; r_1) \overline{\Kl_3(a m_2 l \bar{r}_2^3; r_1)} 
e \left( -\frac{hl}{r_1} \right),
\end{equation}
where $\widehat{V}$ denotes the Fourier transform of $V$. 
The inner complete exponential sum factorizes, and we define
\[
\mathcal{S}(h, m_1, m_2, p) := \sum_{l \mods{p}} \Kl_3(m_1 l ; p) 
\overline{\Kl_3(m_2 l ; p)} e \left( -\frac{hl}{p} \right),
\]
so that
\begin{equation*}
\sum_{l \bmod{r_1}} 
\Kl_3(am_1 l \bar{r}_2^3; r_1) \overline{\Kl_3(a m_2 l \bar{r}_2^3; r_1)} e \left( -\frac{hl}{r_1} \right)= \prod_{j=1}^{J} \mathcal{S}(\bar{Q}_j h,\bar{Q}_j^3 \bar{r}_2^3 m_1, \bar{Q}_j^3 \bar{r}_2^3 m_2, p_j).
\end{equation*}
We have the following bounds for $\mathcal{S}(h, m_1, m_2, p)$.

\begin{lemma}\label{lem:S-bounds}
Let $p$ be an odd prime, and let $h, m_1, m_2 \in \mathbb{Z}$ and $d \in \mathbb{N}$. Then
\[
\mathcal{S}(h, m_1, m_2, p) \ll p^{1/2} (m_1 - m_2, h, p)^{1/2}, 
\]
\end{lemma}

\begin{proof}
Substituting the definition
\[
\Kl_3(mld;p)=\frac{1}{p}\;\sideset{}{^*}
\sum_{x\mods p}e\left(\frac{mlx}{p}\right)
S(1,\bar{x};p),
\]
and executing the $l \mods p$ sum, we obtain
\[
\mathcal{S}(h, m_1, m_2, p) := \frac{1}{p}\;\sideset{}{^*}\sum_{\substack{x,y \mods p 
\\ m_1x-m_2y\equiv l \mods{s}}}S(1,\bar{x};p)S(1,\bar{y};p).
\]
Notice that if we set
\[
\gamma_1 = \gamma_2 = \begin{pmatrix} 0 & 1 \\ 1 & 0 \end{pmatrix},
\qquad
\gamma_3 = \begin{pmatrix} \overline{m_1 } & 0 \\-\overline{h} & \overline{m_2 } \end{pmatrix},
\]
then a straightforward computation shows
\[
\gamma_2 \gamma_3 \gamma_1^{-1} = \begin{pmatrix}
\overline{m_2 } & 0 \\
-\overline{h}     & \overline{m_1 }
\end{pmatrix}
\begin{pmatrix} 0 & 1 \\ 1 & 0 \end{pmatrix}
= \begin{pmatrix}
0 & \overline{m_2 } \\
\overline{m_1 } & -\overline{h}
\end{pmatrix}.
\]
After a change of variable $x \mapsto \gamma_1^{-1}(x)$ (which is permissible since $\det(\gamma_1) \not\equiv 0 \pmod p$), we obtain
\[
\mathcal{S}(h, m_1, m_2, p) = \frac{1}{p}\; 
\sideset{}{^*}\sum_{x \mods p} S\bigl(1, x; p\bigr) \,
\overline{S}\bigl(1, \gamma_2 \gamma_3 \gamma_1^{-1}(x); p\bigr).
\]

We are in a position to apply the estimates from \cite[Propositions 3.3 and 3.4]{DRF1997} amount to the following. Given $\gamma = \begin{pmatrix} a & b \\ c & d \end{pmatrix}$ 
be a matrix with integer entries such that $\det(\gamma) = ad - bc \not\equiv 0 \pmod p$, then
\[
\sideset{}{^*}\sum_{\alpha \mods p} S(1, \alpha; p) \,
\overline{S\bigl(1, \gamma(\alpha); p\bigr)} \ll p^{\frac{3}{2}} 
+ p^{2}\delta_{a-d \equiv b \equiv c \equiv 0 \mods p}.
\]
Now we apply this estimate with the matrix $\gamma_2 \gamma_3 \gamma_1^{-1}$.  
The condition $a-d \equiv b \equiv c \equiv 0 \pmod p$ becomes
\[
0 - (-\overline{h}) \equiv \overline{m_2 } \equiv \overline{m_1 } \equiv 0 \pmod p,
\]
which is equivalent to $h \equiv 0 \pmod p$ and $m_1  \equiv m_2  \pmod p$.   
Therefore, we have the following bound
\[
\mathcal{S}(h, m_1, m_2, p) \ll p \delta_{\substack{h\equiv 0\mods{p}\\ m_1\equiv m_2\mods{p}}}+p^{\frac{1}{2}}.
\]
This gives the stated result.
\end{proof}

We conclude from Lemma \ref{lem:S-bounds} that
\begin{equation*}
\sum_{l \mods{r_1}}\Kl_3(a m_1 l \bar{r}_2^3; r_1) 
 \overline{\Kl_3(a m_2 l \bar{r}_2^3; r_1)} e \left( -\frac{hl}{r_1} \right) 
 \ll r_1^{\frac{3}{2}+\varepsilon} (m_1 - m_2, r_1)^{\frac{1}{2}} \delta_{(m_1 - m_2, r_1) \mid h}.
\end{equation*}
Substituting this back into~\eqref{after-Poisson}, we obtain
\begin{align*}
\Sigma (m_1, m_2, r_1) &\ll  M r_1^{\frac{1}{2}+\varepsilon} (m_1 - m_2, r_1)^{\frac{1}{2}} 
\sum_{\substack{h \\ (m_1 - m_2, r_1) \mid h}} \left| \widehat{V} \left( \frac{hM}{r_1} \right) \right| \\
&\ll M r_1^{\frac{1}{2}+\varepsilon} (m_1 - m_2, r_1)^{\frac{1}{2}} 
\left( 1 + \frac{r_1}{M (m_1 - m_2, r_1)}\right) \\
&\ll M r_1^{\frac{1}{2}+\varepsilon} (m_1 - m_2, r_1)^{\frac{1}{2}} 
+ r_1^{\frac{3}{2}+\varepsilon} (m_1 - m_2, r_1)^{-\frac{1}{2}}.
\end{align*}
In view of the definition of $H$ and~\eqref{q-factorization}, we conclude that
\begin{align*}
&\sum_{M \leqslant m \leqslant 2M} \sum_{0 \neq k \ll \frac{M}{H r_2}} 
\bigl|\Sigma(m, m + k H r_2, r_1) \bigr| \\
&\qquad\ll  r_1^\varepsilon \sum_{M \leqslant m \leqslant 2M}
\sum_{0 \neq k \ll \frac{M}{H r_2}} 
\Bigl( M r_1^{\frac{1}{2}} (k, r_1)^{\frac{1}{2}} H^{\frac{1}{2}} + r_1^{\frac{3}{2}} (k, r_1)^{\frac{1}{2}} H^{-\frac{1}{2}} \Bigr) \\
&\qquad\ll r_1^\varepsilon \frac{M^2}{H r_2} M H^{\frac{1}{2}} r_1^{\frac{1}{2}} 
+ r_1^\varepsilon \frac{M^2}{H r_2} \frac{r_1^{\frac{3}{2}}}{H^{\frac{1}{2}}} .
\end{align*}
Inserting this into~\eqref{Sigma-ll} yields the final bound
\begin{equation}\label{eq:final-bound}
\Sigma \ll q^\varepsilon \biggl( H r_2 M L r_1 + M^2 r_1^{\frac{1}{2}} L (H r_2)^{\frac{1}{2}} 
+ \frac{M^2 r_1^{\frac{3}{2}}}{(H r_2)^{\frac{1}{2}}} \biggr).
\end{equation}

This estimate holds for every decomposition $q = r_1 r_2$ with $(r_1, r_2) = 1$ and for every divisor $H$ satisfying $H \mid r_1$. Note that the right-hand side depends only on the product $H r_2$, not on the individual factors $H$ and $r_2$. This is to be expected conceptually, because the product $H r_2$ acted as a single differencing step in Lemma~\ref{lem:Weyl-differencing}.

We now return to Theorem~\ref{thmBKs}. Given a divisor $s \mid q$, define $r_2$ as the largest divisor of $s$ that is coprime to $(s, \frac{q}{s})$. Set $H =\frac{s}{ r_2}$ and $r_1 = \frac{q}{r_2} $, that is,
\[
H = \Bigl( s,\, \bigl( s, \tfrac{q}{s} \bigr)^\infty \Bigr), \qquad 
r_2 = \frac{s}{H}, \qquad 
r_1 = \frac{q}{r_2}.
\]
Here $(\cdot, \cdot)^\infty$ denotes the largest common divisor at all finite places.
With this choice of $H$ and the decomposition $q = r_1 r_2$, conditions~\eqref{q-factorization} and the requirement $H \mid r_1$ are satisfied. 
This completes the proof of Theorem~\ref{thmBKs}.
\end{proof}

According to Theorem \ref{thmBKs}, we have derived the following corollary.
\begin{corollary}\label{main corollary}
Under the conditions of Theorem \ref{thmBKs}, let q be a squarefree number which has a factor $s\in(q^{\frac{1}{3}-\varepsilon},q^{\frac{1}{3}+\varepsilon})$. For $\alpha_m=(\lambda_f(m))_{m\leqslant M}$, $\beta_l=\{1\}$, we have
\begin{equation}\label{eqBKs}
\begin{aligned}
&\sum_{m \leqslant M} \sum_{l \leqslant L} \lambda_f(m) \mathrm{Kl}_3(a m l; q)\ll q^{\varepsilon} ML\bigl(q^{\frac{1}{6}}M^{-\frac{1}{2}} + q^{-\frac{1}{6}} + q^{\frac{1}{6}}L^{-\frac{1}{2}}\bigr).
\end{aligned}
\end{equation}
\end{corollary}

We establish an estimate for the general bilinear form by using the exponential method as follows.
\begin{lemma}\label{exponent pairs}
    Let $q\geqslant 1$ be a square-free integer, let $a\in\mathbb{Z}$ be coprime to $q$, and let $\alpha_{m}$ and $\beta_l$ be defined as in \eqref{B()}. 
Then we have 
\begin{equation*}
\sum_{m\leqslant M}\sum_{l\in \mathcal{L}}\alpha_m\beta_l\mathrm{Kl}_3(aml;q)\ll q^{\varepsilon}Q^{C}\|\alpha\|_2\|\beta\|_2\big(ML\big)^{\frac{1}{2}}\big(q^{\frac{1}{4}}M^{-\frac{1}{2}}+q^{-\frac{1}{4}}+L^{-\frac{1}{2}}\big),
\end{equation*}
where the implied constant depends only on $\varepsilon$, $Q\geqslant1$, and $C\geqslant0$ is an absolute constant.
\end{lemma}
\begin{proof}
The bound follows from the general theory of bilinear forms with Kloosterman sums developed by Fouvry, Kowalski and Michel \cite[Theorem 1.16]{FKM2014}, which applies to algebraic trace functions. For the specific case of $\mathrm{Kl}_3$, we apply the exponent pair method to estimate the upper bound. If $M<q$, then by \cite[Lemma 2.3]{Xi} one has the estimate
\begin{equation}\label{exponent pair}
\begin{aligned}
    & \sum_{m\leqslant M}\sum_{l\leqslant L}\alpha_m\beta_l K_q(ml)\\
&\qquad \ll_{\varepsilon}\big(qML\big)^{\varepsilon}\big(ML\big)^{\frac{1}{2}}\|\alpha\|_2\|\beta\|_2\big(q^{\frac{\kappa}{2}}M^{\frac{\lambda-\kappa-1}{2}}+q^{-\frac{1}{4}}+L^{-\frac{1}{2}}\big).
\end{aligned}
\end{equation}

For the proof of Theorem \ref{main theorem}, we use the exponent pair $(\kappa, \lambda)=(\frac{1}{2},\frac{1}{2})$, in which case $q$ is not required to have only a small prime factor.
Then the described bound is obtained. 
\end{proof}
The upper bound estimate for the special bilinear form studied from Sharma \cite{Sharma} is as follows:

\begin{lemma}\label{1_L}
Under the assumptions of Corollary \ref{exponent pairs}, let $\alpha_m=(\lambda_f(m))_{m\leqslant M}$ and $\mathcal{L}=\{1\}$. Then we have
\begin{equation*}
\sum_{m\leqslant M}\lambda_f(m)\mathrm{Kl}_3(am;q)\ll q^{\varepsilon}M(q^{-\frac{1}{4}}+q^{\frac{3}{8}}M^{-\frac{1}{2}}+q^{\frac{3}{4}}M^{-1}),
\end{equation*}
where the implied constant depends only on $\varepsilon$.
\end{lemma}
\begin{proof}
See \cite[(1.8)]{Sharma}.
\end{proof}

\section{Proof of Theorem \ref{main theorem}}
\subsection{Auxiliary lemmas}
To prove Theorem \ref{main theorem}, we will need several lemmas.
First, we recall a form of the Poisson summation formula in arithmetic progressions.

\begin{lemma}
For any positive integer $q\geqslant 1$, any function $K$ that is $q$-periodic, and any smooth function $V$ compactly supported on $\mathbb{R}$, we have
\begin{equation*}
\sum_{n\geqslant1}K(n)V(n)=\frac{1}{\sqrt{q}}\sum_{m\geqslant1}
\widehat{K}(m)\widehat{V}(\frac{m}{q}),
\end{equation*}
and in particular
\begin{equation*}
\sum_{\substack{n\geqslant1\\n\equiv a\mods{q}}}V(n)=\frac{1}{q}\sum_{m\geqslant1}e\Bigl{(}\frac{am}{q}\Bigl{)}\widehat{V}\Bigl{(}\frac{m}{q}\bigl{)}.
\end{equation*}
\end{lemma}

Then, we recall the $\operatorname{GL}(2)$ Voronoi summation formula in the following lemma. See appendix A.4 of \cite{Vonoroi1} and appendix of \cite{Vonoroi2} for details.
\begin{lemma}\label{Voronoi summation formula}
Let $f$ be a holomorphic cusp form with weight $k$, and let $\lambda_f(n)$ be its Fourier coefficients. For integers $a, q\geqslant1$ with $(a,q)=1$, and $V(x)\in C_c^\infty(0,+\infty)$, we have
\begin{equation*}
\sum_{n\geqslant1}\lambda_f(n)e\Bigl{(}\frac{an}{q}\Bigl{)}V(n)=\frac{1}{q}\sum_{n\geqslant 1}\lambda_f(n)e\left(\frac{\overline{a}n}{q}\right)\check{V}\left(\frac{n}{q^2}\right),
\end{equation*}
where $\check{V}$ is the Bessel transform of weight $k$ given by
\begin{equation}\label{V}
    \check{V}(x)=2\pi i^{k}\int_0^{\infty}V(t)J_{k-1}(4\pi \sqrt{xt})\mathrm{d}t.
\end{equation}
\end{lemma}

\begin{lemma}\label{proposition2}
Let $q$ be a squarefree number and $d\mid q$. Let $V$, $W$ be two smooth functions compactly supported on $[0, +\infty)$. Then we have
\begin{equation*}
    \begin{aligned}
\mathop{\sum\sum}_{\substack{m, l\geqslant1\\ml\equiv a\mods q}}\lambda_f(m)V(m)W(l)=&\frac{1}{q}
\mathop{\sum\sum}_{\substack{m, l\geqslant1\\(ml,q)=1}}\lambda_f(m)V(m)W(l)
\\&+\frac{1}{q}\sum_{d\mid q}\frac{1}{d}\mathop{\sum\sum}_{m,l\geqslant1}\lambda_{f}(m)\check{V}\Bigl{(} \frac{m}{d^{2}}\Bigr{)}\widehat{W}\Bigl{(}\frac{l}{d}\Bigr{)}S_3(m,l,a;d).
\end{aligned}
\end{equation*}
where $\check{V}$ is defined as in \eqref{V}, $\widehat{W}$ denotes the Fourier transform of $W$, and $S_3(m,l,a;d)$ is the $3$th Kloosterman sum. 
\end{lemma}

\begin{proof}
We are interested in the asymptotic of 
\begin{equation*}
\mathop{\sum\sum}_{\substack{m,l\geqslant1\\ ml\equiv a \mods q}}\lambda_f(m)V(m)W(l).
\end{equation*}
Detecting $ml\equiv a \mods q$ using additive characters, we obtain
\begin{equation*}
\frac{1}{q}\mathop{\sum\sum}_{m,l\geqslant1}
\sum_{u \mods q}e\left(\frac{u(a-ml)}{q}\right)\lambda_f(m)V(m)W(l).
\end{equation*}
Splitting into Ramanujan sums, we get
\begin{equation}\label{d>1}
\frac{1}{q}\sum_{d\mid q}\mathop{\sum\sum}_{\substack{m,l\geqslant 1\\ (l,q)=1}}\;\sideset{}{^*}\sum_{u \mods d}e\left(\frac{u(a-ml)}{d}\right)\lambda_f(m)V(m)W(l).
\end{equation}
Trivially, the contribution of $d=1$ is
\begin{equation*}
\frac{1}{q}\mathop{\sum\sum}_{\substack{m, l\geqslant1\\(ml,q)=1}}\lambda_f(m)V(m)W(l).
\end{equation*}
When $d>1$, we first consider the $m$-sum.
Applying the Voronoi summation formula to the above sum over $m$ and 
according to Lemma \ref{Voronoi summation formula}, we obtain
\begin{equation*}
\sum_{m\geqslant 1}\lambda_{f}(m)V(m)e\Bigl(-\frac{uml}{d}\Bigr)
=\frac{1}{d}\sum_{m\geqslant 1}\lambda_{f}(m)\check{V}
\Bigl(\frac{m}{d^{2}}\Bigr)e\Bigl(\frac{ \overline{l}\overline{u}m}{d}\Bigr).
\end{equation*}
Therefore, \eqref{d>1} is
\begin{equation}\label{S2/}
    \frac{1}{q}\sum_{d\mid q}\frac{1}{d}\mathop{\sum\sum}_{\substack{m,l\geqslant 1\\ (l,d)=1}}\lambda_{f}(m)\check{V}\Bigl{(}\frac{m}{d^{2}}\Bigr{)}W(l)
    \,\sideset{}{^*}
\sum_{u \mods d}e\Bigl(\frac{au+m \overline{l}\overline{u}}{d}\Bigl)
\end{equation}
Applying the Poisson summation formula to the sum over \(l\), we obtain
\begin{equation*}
\sum_{\substack{l\geqslant1\\(l,d)=1}}W(l)e\Bigl{(}\frac{m \overline{l}\overline{u}}{d}\Bigr)=\frac{1}{d}
\;\sideset{}{^*}\sum_{v \mods d}\sum_{l\geqslant 1}
\widehat{W}\Bigl(\frac{l}{d}\Bigr)e\Bigl(\frac{m\overline{u}\overline{v}+lv}{d}\Bigr).
\end{equation*}
So \eqref{S2/} becomes
\begin{equation*}
\frac{1}{q}\sum_{d\mid q}\frac{1}{d^{2}}\mathop{\sum\sum}_{m,l\geqslant1}\lambda_{f}(m)\check{V}\Bigl{(} \frac{m}{d^{2}}\Bigr{)}\widehat{W}\Bigl{(}\frac{l}{d}\Bigr{)}\sumstar_{u\mods d}e\Bigl{(}{\frac{au}{d}}\Bigl{)}\sumstar_{v\mods d}e\Bigl{(}\frac{m\overline{uv}+lv}{d}\Bigl{)}.
\end{equation*}
Then, we get
\begin{equation*}
\frac{1}{q}\sum_{d\mid q}\frac{1}{d}\mathop{\sum\sum}_{m,l\geqslant1}\lambda_{f}(m)\check{V}\Bigl{(} \frac{m}{d^{2}}\Bigr{)}\widehat{W}\Bigl{(}\frac{l}{d}\Bigr{)}S_3(m,l,a;d).
\end{equation*}

This gives the formula we stated.
\end{proof}

We take $\Delta$ slightly larger than $1$, that is,
\begin{equation*}
   \Delta= 1 + (\log X)^{-B},
\end{equation*}
where \(B \geqslant 1\) is sufficiently large. Given \(X \geqslant 2\). Then we introduce the following lemma.
\begin{lemma}
    For any $\Delta>1$, there exists a sequence $(b_{l,\Delta})_{l \geqslant 0}$ of smooth functions
supported in $[\Delta^{l-1},\Delta^{l+1}]$ such that, for all $t \geqslant 1$ and $v \geqslant 0$,
\[
b^{(v)}_{l,\Delta}(t) \ll_v t^{-v} (\log X)^{Bv}.
\]
\begin{proof}
    See \cite[Lemma 2]{Fouvry}.
\end{proof}
\end{lemma}
We define $V(t)=b_{l,\Delta}(t), W(t)=b_{l,\Delta}(t)$. Thus, their derivatives satisfy\begin{equation}\label{V'}
    V^{(v)}(t), W^{(v)}(t)\ll_v t^{-v}(\log X)^{Bv}.
\end{equation}

The following lemma describes the decay of the Fourier transforms of $W$ and the Bessel transforms of $V$.
\begin{lemma}\label{decay}
Let $V, W$ be as above, $\check{V}$ be defined as in \eqref{V}, and $\widehat{W}$ denote the Fourier transform of $W$. There exists a constant $D \geqslant 0$ such that for any $t > 0$, any $E \geqslant 0$ and any $j \geqslant 0$, we have
\begin{align*}
\check{V}^{(j)}(t) &\ll_{E,j} t^{-j}M(\log X)^{Bj} \Bigl{(}\frac{(\log X)^{Dj}}{1+tM}\Bigl{)}^{E},  \\
\hat{W}^{(j)}(t) &\ll_{E, j} t^{-j}L(\log X)^{Bj} \Bigl{(}\frac{(\log X)^{Dj}}{1+tL}\Bigl{)}^{E}. 
\end{align*}
\begin{proof}
    See \cite[Lemma 5.3]{KMS2017}.
\end{proof}

\end{lemma}
\subsection{Reduction of Theorem \ref{main theorem}}
We will set up the proof of Theorem \ref{main theorem} in a way very similar to the method introduced in \cite{KMS2017}. We define 
\begin{equation*}
E(X;q,a):=\sum_{\substack{n\leqslant X\\n\equiv a\mods{q}}}(\lambda_f*1)(n)-\frac{1}{\varphi(q)}\sum_{\substack{n\leqslant X\\(n,q)=1}} (\lambda_f*1)(n).
\end{equation*}
From \cite{FKM2015}, by a smooth partition of unity on the \(m\) and \(l\) variables, and decompose $E(\lambda_f*1,x;q,a)$ into $O(\log^{2} X)$ terms of the form
\begin{equation}\label{e}
\widetilde{E}(V, W; q, a)=\sum_{ml\equiv a \mods{q}} \lambda_{f}(m) V(m)W(l) - \frac{1}{q} \sum_{(ml, q) = 1} \lambda_{f}(m) V(m)W(l),
\end{equation}
where \(V, W\) satisfies \eqref{V'}.
Applying Lemma \ref{proposition2} to \eqref{e}, we obtain
\begin{equation}\label{E}
\widetilde{E}(V, W;q,a)
=\frac{1}{q}\sum_{d\mid q}\frac{1}{d}\mathop{\sum\sum}_{m,l\geqslant1}\lambda_{f}(m)\check{V}\Bigl{(} \frac{m}{d^{2}}\Bigr{)}\widehat{W}\Bigl{(}\frac{l}{d}\Bigr{)}S_3(m,l,a;d).
\end{equation}
From \cite[Lemma 2.1]{Xi}, we have
\begin{equation*}
S_3(m,l,a;d)=\sum_{s\mid(m,l,d )}\mu(s)S_3(\frac{\overline{s}mla}{s^{2}},1,1;\frac{d}{s}).
\end{equation*}
Then, \eqref{E} becomes
\begin{equation*}
\begin{aligned}
&\frac{1}{q}\sum_{d\mid q}\frac{1}{d}\Bigl{|}\mathop{\sum\sum}_{m,l\geqslant1}\lambda_{f}(m)\check{V}\Bigl{(} \frac{m}{d^{2}}\Bigr{)}\widehat{W}\Bigl{(}\frac{l}{d}\Bigr{)}\sum_{s\mid (m,l,d)}\mu(s)S_3(\frac{\overline{s}mla}{s^{2}},1,1;\frac{d}{s})\Bigl{|}\\
&\ll\frac{1}{q}\sum_{d\mid q}\sum_{s\mid (m,l,d)}\frac{1}{d}\Bigl{|}\mathop{\sum\sum}_{m,l\geqslant1}\lambda_{f}(m)\check{V}\Bigl{(} \frac{m}{d^{2}}\Bigr{)}\widehat{W}\Bigl{(}\frac{l}{d}\Bigr{)}\mathrm{Kl}_3(\frac{\overline{s}mla}{s^{2}};\frac{d}{s})\Bigl{|}
\end{aligned}
\end{equation*}

Hence, it remains to prove that
\begin{equation}\label{简化估计}
    \frac{1}{q}\sum_{d\mid q}\sum_{s\mid (m,l,d)}\frac{1}{d}\Bigl{|}\mathop{\sum\sum}_{m,l\geqslant1}\lambda_{f}(m)\check{V}\Bigl{(} \frac{m}{d^{2}}\Bigr{)}\widehat{W}\Bigl{(}\frac{l}{d}\Bigr{)}\mathrm{Kl}_3(\frac{\overline{s}mla}{s^{2}};\frac{d}{s})\Bigl{|} \ll_{A} \frac{ML}{q} (\log X)^{-A},
\end{equation}
where $X(\log X)^{-C} \leqslant ML\leqslant X$
for some \(C \geqslant 0\) large enough depending on the value of the parameter $A$ in Theorem \ref{main theorem}.
We introduce the dual parameters
\begin{equation}\label{no1}
M^*=\frac{q^{2}}{M}, \quad L^*=\frac{q}{L}.
\end{equation}
By Lemma \ref{decay}, we know that if \(\eta>0\) is sufficiently small, the contribution to the sum \eqref{简化估计} from those pairs \((m,l)\) such that
$m\geqslant X^{\eta/2}M^{*}$ and $l\geqslant X^{\eta/2}L^{*}$ is negligible. Therefore, by Lemma \ref{decay} and a smooth dyadic partition of unity, we reduce the estimation of the sum to the form
\begin{equation*}
S(M^{\prime},L^{\prime})=\sum_{m,l\geqslant 1}\lambda_{f}(m)\mathrm{Kl}_{3}(aml;q)V^{*}(m)W^{*}(l),
\end{equation*}
where
\begin{equation}\label{no2}
    \frac{1}{2}\leqslant M^{\prime}\leqslant X^{\frac{\eta}{2}}M^{*},\quad \frac{1}{2}\leqslant L^{\prime} \leqslant X^{\frac{\eta}{2}}L^{*},
\end{equation}
and \(V^{*},W^{*}\) are smooth compactly supported functions with
\begin{equation*}
\operatorname{supp}(V^{*})\subset[M^{\prime},2M^{\prime}],\quad\operatorname{supp}(W^{*})\subset[L ^{\prime},2L^{\prime}],
\end{equation*}
\begin{equation*}
V^{*(j)}(u),\quad W^{*(j)}(u)\ll u^{-j} (\log X)^{O(j)}
\end{equation*}
for any \(j\geqslant 0\). Precisely, it is enough to prove that
\begin{equation*} S(M^{\prime},L^{\prime})\ll_{A}q(\log X)^{-A}.
\end{equation*}
Since the trivial bound for $S(M^{\prime},L^{\prime})$ is
\begin{equation*}
S(M^{\prime},L^{\prime})\ll M^{\prime}L^{\prime}(\log X).
\end{equation*}

\subsection{Completion of the proof of Theorem \ref{main theorem}}\label{end proof}
First of all, let us introduce some notation. We denote
\begin{equation*}
    X=q^{2-\delta}
\end{equation*}
and
\begin{equation*}
    M=q^{\mu},\quad L=q^{\nu},\quad M^{\prime}=q^{\mu^{\prime }},\quad L^{\prime}=q^{\nu^{\prime}},\quad M^{*}=q^{\mu^{*}},\quad L^{*}=q^{\nu^{*}}.
\end{equation*}
By \eqref{no1} and \eqref{no2}, we have
\begin{equation*}
    \mu^{*}=2-\mu,\quad\nu^{*}=1-\nu,\quad\mu^{\prime}\leqslant\mu^{*}+\frac{\eta}{2}, \quad\nu^{\prime}\leqslant\nu^{*}+\frac{\eta}{2}.
\end{equation*}
In addition,
\begin{equation*}
    \mu+\nu=2-\delta+o(1), 
\end{equation*}
so we get
\begin{equation*}
\mu^{\prime}+\nu^{\prime}\ll1+\delta+\eta+o(1).
\end{equation*}
Let
\begin{equation*}
S(M^{\prime},L^{\prime})=q^{\sigma(\mu^{\prime},\nu^{\prime})}.
\end{equation*}
Then Corollary \ref{main corollary}, Lemma \ref{exponent pairs} and Lemma \ref{1_L} yield the estimates
\begin{equation*}
\sigma(\mu^{\prime},\nu^{\prime})\leqslant\tau(\mu^{\prime},\nu^{\prime})+o(1),
\end{equation*}
where
\begin{align}
\tau(\mu^{\prime},\nu^{\prime}) &\leqslant \mu^{\prime}+\nu^{\prime}+ \max\Bigl{(}\frac{1}{6}-\frac{\mu^{\prime}}{2}, -\frac{1}{6}, \frac{1}{6}-\frac{\nu^{\prime}}{2}\Bigl{)},\label{bound1}\\
\tau(\mu^{\prime},\nu^{\prime}) &\leqslant \mu^{\prime}+\nu^{\prime}+ \max\Bigl{(}\frac{1}{4}-\frac{\mu^{\prime}}{2}, -\frac{1}{4}, -\frac{\nu^{\prime}}{2}\Bigl{)}, \label{bound2}\\
\tau(\mu^{\prime},\nu^{\prime}) &\leqslant \mu^{\prime}+\nu^{\prime}+ \max\Bigl{(}-\frac{1}{4},\frac{3}{8}-\frac{\mu^{\prime}}{2},\frac{3}{4}-\mu^{\prime}\Bigl{)}.\label{bound3}
\end{align}

We will prove that if $\delta<\frac{1}{18}$ and $\eta$ is small enough, then we have 
\begin{equation*}
\sigma(\mu^{\prime},\nu^{\prime})\leqslant 1-\kappa,
\end{equation*}
where $\kappa>0$, sufficiently small, depends only on $\delta$ and $\eta$. This implies the desired estimate.

First, since
\begin{equation*}
 \mu^{\prime}+\nu^{\prime}-\frac{1}{4}\leqslant \mu^{\prime}+\nu^{\prime}-\frac{1}{6}\leqslant 1+\delta-\frac{1}{6}+o(1)<1, 
\end{equation*}
we may replace \eqref{bound1}, \eqref{bound2} and \eqref{bound3} by
\begin{align}
\tau(\mu^{\prime},\nu^{\prime}) &\leqslant \mu^{\prime}+\nu^{\prime}+ \max\Bigl{(}\frac{1}{6}-\frac{\mu^{\prime}}{2}, \frac{1}{6}-\frac{\nu^{\prime}}{2}\Bigl{)},\label{boundd1}\\
\tau(\mu^{\prime},\nu^{\prime}) &\leqslant \mu^{\prime}+\nu^{\prime}+ \max\Bigl{(}  -\frac{\nu^{\prime}}{2}, \frac{1}{4}-\frac{\mu^{\prime}}{2}\Bigl{)},\label{boundd2}\\
\tau(\mu^{\prime},\nu^{\prime}) &\leqslant \mu^{\prime}+\nu^{\prime}+ \max\Bigl{(}\frac{3}{8}-\frac{\mu^{\prime}}{2},\frac{3}{4}-\mu^{\prime}\Bigl{)} \label{boundd3}
\end{align}
We now distinguish three cases according to the relative sizes of $\mu^{\prime}$ and $\nu^{\prime}$:
\begin{itemize}
    \item if $\mu^{\prime}\leqslant\frac{1}{4}-\kappa$, we replace \eqref{boundd3} by
    \begin{equation}
    \tau(\mu^{\prime},\nu^{\prime}) \leqslant \frac{\mu^{\prime}+\nu^{\prime}}{2}+\frac{\nu^{\prime}}{2}+\frac{3}{8}\label{bounddd1};
    \end{equation}
    \item if $\mu^{\prime}>\frac{1}{2}+2(\delta+\kappa)$ and $\nu^{\prime}>2(\delta+\kappa)$, then we obtain the desired bound from \eqref{boundd2};
    \item if $\nu^{\prime}\leqslant2(\delta+\kappa)$, then, using \eqref{bounddd1} and choosing $\kappa$ sufficiently small, we also obtain the desired bound, since in that case
\begin{align*}
\tau(\mu^{\prime},\nu^{\prime}) \leqslant\frac{1+\delta+o(1)}{2}+\delta+\kappa+\frac{3}{8}\leqslant1-\frac{3}{2}\Bigl{(}\frac{1}{12}-\delta\Bigl{)}+\kappa+o(1);
\end{align*}
\item finally, if $\frac{1}{4}-\kappa<\mu^{\prime}\leqslant\frac{1}{2}+2(\delta+\kappa)$, using \eqref{boundd1}, similarly we have
\begin{equation*}
\tau(\mu^{\prime},\nu^{\prime}) \leqslant 1-\frac{3}{2}\Bigl{(}\frac{1}{18}-\delta\Bigl{)}+\kappa+o(1).
\end{equation*}
\end{itemize}

This completes the proof of Theorem \ref{main theorem} for any sufficiently small $\eta$ satisfying
\begin{equation*}
    X^{\frac{1}{2}-\eta}\leqslant q\leqslant X^{\frac{18}{35}-\eta},
\end{equation*}
we have 
\begin{equation*}
E(X;q,a)\ll_{\varepsilon, A}\frac{X}{q}(\log X)^{-A}.
\end{equation*}
	
\section*{Acknowledgements}
The authors thank Professor Yongxiao Lin and Professor Ping Xi for their helpful comments and suggestions.

\bibliographystyle{plain}
\bibliography{distribution/references}  % 你的 .bib 文件名
\end{document}